\documentclass[12pt]{amsart}
\usepackage{microtype}
\usepackage{etex} 
\usepackage[active]{srcltx}

\usepackage{calc,amssymb,amsthm,amsmath,amscd, eucal,ulem, mathtools}
\usepackage{phaistos}
\usepackage{alltt}
\RequirePackage[dvipsnames,usenames]{color}

\input{mabliautoref.sty}
\input{xy}
\xyoption{all}
\input{kmacros3.sty}
\usepackage{tikz,calc}
\usepackage{tikz-cd}
\usepackage[marvosym]{tikzsymbols}
\usepackage{amsfonts, mathrsfs}
\usepackage[cal=boondox, calscaled=1.05]{mathalfa}
\usepackage{calligra}
\usepackage{stmaryrd} 
\usepackage{dcpic, pictexwd}
\usepackage[left=1in,top=1in,right=1in,bottom=1in]{geometry}
\usepackage{bm}
\usepackage{verbatim}
\usepackage{upgreek}

\numberwithin{equation}{theorem}

\newcommand{\vp}{\varphi}

\renewcommand{\:}{\colon}

\newcommand{\kay}{\mathcal{k}}

\DeclareMathOperator{\Cl}{Cl}

\DeclareMathOperator{\ssHom}{ \sH\!\!\;\!\!\text{\calligra{\Large om}\,}}

\newcommand{\exes}[1]{\bm{x}^{ {#1}_{\bullet} }}
\newcommand{\whys}[1]{\bm{y}^{ {#1}_{\bullet} }}

\newcommand{\exesi}[2]{\bm{x}_{#2}^{ {#1}_{\bullet, #2} }}
\newcommand{\whysi}[2]{\bm{y}_{#2}^{ {#1}_{\bullet, #2} }}
\newcommand{\zeesi}[2]{\bm{z}_{#2}^{ {#1}_{\bullet, #2} }}

\newcommand{\exestothe}[1]{\bm{x}^{#1}}
\newcommand{\whystothe}[1]{\bm{y}^{#1}}

\newcommand{\sumofas}[1]{\sum_{k} a_{#1, k}}
\newcommand{\sumofbs}[1]{\sum_{k} b_{#1, k}}

\newcommand{\diagCA}[1]{\sD^{(#1)}}

\DeclarePairedDelimiter\floor{\lfloor}{\rfloor}

\newcommand{\liftofphi}{\widehat \varphi_e}

\usepackage{setspace}
\usepackage{hyperref}
\usepackage{enumerate}
\usepackage{graphicx}
\usepackage[all,cmtip]{xy}

\usepackage{verbatim}

\theoremstyle{theorem}

\def\todo#1{\textcolor{Mahogany}%
{\footnotesize\newline{\color{Mahogany}\fbox{\parbox{\textwidth-15pt}{\textbf{todo:
} #1}}}\newline}}

\renewcommand{\O}{\mathcal O}
\renewcommand{\sO}{\mathcal{O}}
\renewcommand{\sF}{\mathcal{F}}
\renewcommand{\sE}{\mathcal{E}}
\renewcommand{\sC}{\mathcal{C}}
\renewcommand{\sD}{\mathcal{D}}
\renewcommand{\sI}{\mathcal{I}}
\renewcommand{\sH}{\mathcal{H}}
\renewcommand{\sG}{\mathcal{G}}
\usepackage{marvosym}

\begin{document}
\title[USTP for Diagonally $F$-Regular Algebras]{The Uniform Symbolic Topology Property for Diagonally $F$-regular Algebras} 
\author[J.~Carvajal-Rojas]{Javier Carvajal-Rojas}
\address{Department of Mathematics\\ University of Utah\\ Salt Lake City\\ UT 84112\\USA \newline\indent
Escuela de Matem\'atica\\ Universidad de Costa Rica\\ San Jos\'e 11501\\ Costa Rica}
\email{\href{mailto:carvajal@math.utah.edu}{carvajal@math.utah.edu}}
\author[D.~Smolkin]{Daniel Smolkin}
\address{Department of Mathematics\\ University of Utah\\ Salt Lake City\\ UT 84112\\USA} 
\email{\href{mailto:smolkin@math.utah.edu}{smolkin@math.utah.edu}}

\thanks{The first named author was supported in part by the NSF FRG Grant DMS \#1265261/1501115. The second named author was supported by the NSF grants DMS \#1246989/\#1252860/\#1265261.}



\begin{abstract}
Let $\kay$ be a field of positive characteristic. Building on the work of \cite{SmolkinSubadditivity}, we define a new class of $\kay$-algebras, called \emph{diagonally $F$-regular} algebras, for which the so-called \emph{Uniform Symbolic Topology Property} (USTP) holds effectively. We show that this class contains all essentially smooth $\kay$-algebras. We also show that this class contains certain singular algebras, such as the affine cone over $\mathbb{P}^r_{\kay} \times \mathbb{P}^s_{\kay}$, when $\kay$ is perfect. 
By reduction to positive characteristic, it follows that USTP holds \emph{effectively} for the affine cone over $\bP^r_{\bC} \times \bP^s_{\bC}$ and more generally for \emph{complex varieties of diagonal $F$-regular type}.
\end{abstract}
\maketitle
\section{Introduction}

We are concerned with the following question: when does a finite-dimensional Noetherian ring $R$ satisfy
\begin{equation} \label{eqn.SymbolicOrdinary}
\mathfrak{p}^{(hn)} \subset \mathfrak{p}^n \quad  \forall n \in \bN .
\end{equation}
for all prime ideals $\mathfrak p \subset R$ and for some $h$ independent of $\mathfrak p$? Here, the expression $\mathfrak p^{(m)}$ denotes the $m$-th \emph{symbolic power} of $\mathfrak p$. We invite the reader to glimpse at \cite{DaoDestefaniGrifoHunekeSymbolicSurvey} for an excellent survey on this beautiful but tough problem.

This story starts, perhaps, with the work of I.~Swanson. Swanson established that if $R$ is a Noetherian ring and $\mathfrak{p} \subset R$ is a prime ideal such that the $\mathfrak{p}$-adic and symbolic topologies are equivalent, then they are in fact \emph{linearly equivalent}, meaning there is a constant $h \in \bN$ depending on $\mathfrak{p}$ such that $\mathfrak p^{(hn)} \subset \mathfrak p^n$ for all $n$ \cite{SwansonSymbolicPowers}. 
In particular, Swanson's result holds for every prime ideal $\mathfrak p$ when $R$ is a normal domain essentially of finite type over a field.

Later, Ein--Lazarsfeld--Smith demonstrated in their seminal work \cite{EinLazSmithSymbolic} that if $R$ is a regular $\bC$-algebra essentially of finite type, then $h$ can be taken independently of $\mathfrak{p}$. In fact $h=\dim R$ suffices. Rings for which this number $h$ can be taken independently of $\mathfrak p$ (\ie for which there exists a uniform bound on $h$ for all $\mathfrak p$) are said to have the \emph{Uniform Symbolic Topology Property}, or USTP for short. Ein--Lazarsfeld--Smith's result is now known to hold for any finite-dimensional regular ring: their result was extended to regular rings of equal characteristic by M.~Hochster and C.~Huneke \cite{HochsterHunekeComparisonOfSymbolic} to regular rings of mixed characteristic by L.~Ma and K.~Schwede \cite{MaSchwedePerfectoid}.

Since then, it has been of great interest to know which non-regular rings have USTP. For instance, Huneke--Katz--Validashti showed that, under suitable hypotheses, rings with isolated singularities have USTP, although without an \emph{effective} bound on $h$ \cite{HunekeKatzValidashtiUniformEquivIsolatedSing}. R.~Walker showed that 2-dimensional rational singularities have USTP and obtained an effective bound for $h$ \cite{WalkerRationalSingSymbolicPowers}. 

In this paper, we continue the above efforts in the \emph{strongly $F$-regular} setting. Strong $F$-regularity is a weakening of regularity defined for rings of positive characteristic. Strong $F$-regularity is well-studied by positive characteristic commutative algebraists; see \cite{HunekeTightClosureBook,SchwedeTuckerTestIdealSurvey, SmithZhang}. Given a a field $\kay$ of positive characteristic, we introduce a class of strongly $F$-regular $\kay$ algebras essentially of finite type, called \emph{diagonally $F$-regular} $\kay$-algebras, that are engineered to have USTP; in particular, \autoref{eqn.SymbolicOrdinary} holds for these rings with $h$ equal to dimension. We prove that this class includes all essentially smooth\footnote{If $\kay$ is perfect, then essentially smooth $\kay$-algebras are the same as regular $\kay$-algebras essentially of finite type.} $\kay$-algebras, as well as Segre products of polynomial rings over $\kay$,\footnote{Also known as ``Cartesian products,'' as in \cite[Ch. II, Exc. 5.11]{Hartshorne}.} \ie the affine cone over $\bP^r_{\kay} \times \bP^s_{\kay}$ whenever $\kay$ is a perfect field of positive characteristic and $r,s \geq 1$. We also show that the class of diagonally $F$-regular $\kay$-algebras contains some non-isolated singularities. 

To motivate our approach, we summarize the method introduced in \cite{EinLazSmithSymbolic}, following the presentation of K.~Schwede and K.~Tucker in their survey, \cite[\S 6.3]{SchwedeTuckerTestIdealSurvey}; see also \cite{SmithZhang} by K.~Smith and W.~Zhang.\footnote{Ein--Lazarsfeld--Smith's original argument uses \emph{multiplier ideals}, which are only known to exist in characteristic 0. Their argument was adapted to positive characteristic rings by N.~Hara  \cite{HaraACharacteristicPAnalogOfMultiplierIdealsAndApplications} and to mixed-characteristic rings by L.~Ma and K.~Schwede \cite{MaSchwedePerfectoid}. Hara and Ma--Schwede achieved this by using positive characteristic and mixed characteristic analogs of multiplier ideals, respectively.} 
We do this with the aim of pointing out exactly where this argument breaks down for non-regular rings. In positive characteristic, the crux of Ein--Lazarsfeld--Smith's argument is the following chain of containments:
\begin{equation} \label{eqn.Intro}
  \mathfrak p^{(hn)} \overset{(1)}{\subset}
  \uptau \Big(\mathfrak p^{(hn)}\Big) = 
  \uptau\left( \Big(\mathfrak p^{(hn)}\Big)^{n/n} \right) \overset{(2)}{\subset}
  \uptau\left( \Big(\mathfrak p^{(hn)}\Big)^{1/n} \right)^n \overset{(3)}{\subset}
  \mathfrak p^n
\end{equation}
Here, $\uptau(\mathfrak a^t)$ denotes the test ideal of the (formal) power $\mathfrak a^t$; see \autoref{secn.prelim} for details. Containment (1) holds in any strongly $F$-regular ring. Containment (2) holds by the subadditivity theorem for test ideals---this theorem requires the ambient ring $R$ to be regular. Containment (3) holds quite generally (for $h = \dim R$), as we shall discuss in the proof of \autoref{thm.MainTheorem}.

So, in order to apply this technique to the non-regular case, we must deal with containment (2). Our approach here is simple: we will find an ideal $\mathfrak{t}$, depending on $\mathfrak p$, $h$,  and $n$, such that the second containment
\[
 \mathfrak{t} \overset{(2')}{\subset} \uptau\left( \Big(\mathfrak p^{(hn)}\Big)^{1/n} \right)^n
\]
is guaranteed to hold. Then the problem of deciding whether a particular $F$-regular ring satisfies USTP is reduced to deciding whether the first containment, 
\[
  \mathfrak p^{(nh)} \overset{(1')}{\subset} \mathfrak{t},
\]
holds for our choice of $\mathfrak{t}$. Following \cite{SmolkinSubadditivity}, we will construct $\mathfrak{t}$ using the so-called \emph{diagonal Cartier algebras}. Namely, we set
\[
  \mathfrak{t} = \uptau\Big(\sD^{(n)}; \mathfrak p^{(hn)} \Big),
\]
where $\sD^{(n)}$ is the \emph{$n$-th diagonal Cartier algebra}; see \autoref{def.diagCartAlg}. Then \autoref{prop.basicProperties}(c) demonstrates that containment ($2'$) holds for any reduced $\kay$-algebra essentially of finite type, while ($1'$) holds whenever $\sD^{(n)}$ is $F$-regular. When this is the case for all $n$, we say our ring is \emph{diagonally $F$-regular} as a $\kay$-algebra. This sketches the proof of our main theorem:

\begin{theoremA*}[\autoref{thm.MainTheorem}]
If $R$ is a diagonally $F$-regular $\kay$-algebra essentially of finite type, then $R$ has USTP with $h=\dim R$.
\end{theoremA*}

As we shall see, every essentially smooth $\kay$-algebra is diagonally $F$-regular, but not conversely. Indeed, we have the following:

\begin{theoremB*}[\autoref{thm.SegreDFR}]
Let $\kay$ be a perfect field of positive characteristic, and let $r, s \geq 1$ be integers. Then the affine cone over $\bP^r_\kay \times \bP^s_\kay$ is diagonally $F$-regular.
\end{theoremB*}

Of course, the affine cone over $\bP^r_\kay \times \bP^s_\kay$ is an isolated singularity, and so USTP is known to hold for this ring by \cite{HunekeKatzValidashtiUniformEquivIsolatedSing}. Nonetheless, our result has the virtue of being effective in the sense that we determine the number $h$ explicitly, and show $h$ is as small as we might expect it to be. We also observe that the class of diagonally $F$-regular $F$-finite $\kay$-algebras is closed under tensor products over $\kay$:
\begin{theoremC*}
Let $R$ and $S$ be $\kay$-algebras essentially of finite type, where $\kay$ is a field of characteristic $p$. If $R$ and $S$ are diagonally $F$-regular, then so is $R \otimes_{\kay} S$.
\end{theoremC*}
This implies, remarkably, that the class of diagonally $F$-regular singularities includes some non-isolated singularities. To our knowledge, this gives a new class of examples where USTP is known to hold. We note that R.~Walker obtains orthogonal results to \autoref{thm.MainTheorem} and \autoref{thm.SegreDFR} using complementary techniques; see  \cite{WalkerUniformSymbolicNormalToric,WalkerUniformSymbolicMultinomial} for precise statements.

Finally, let $K$ be a field of characteristic $0$ and $R$ a $K$-algebra. Suppose that $A\subset K$ is a finitely generated $\bZ$-algebra and $R_A \subset R$ an $A$-module such that $A \to R_A$ is a reduction of $K \to R$ \cite[\S 2]{HochsterHunekeTightClosureInEqualCharactersticZero}. We define $R$ to have \emph{diagonally $F$-regular type} if $R_A \otimes A/\mu$ is diagonally $F$-regular for all maximal ideals $\mu$ in a dense open set of $\Spec A$, for all choices of $A$. By standard reduction-mod-$p$ techniques, we get

\begin{theoremD*}[\autoref{lemma.diagFregType}]
Let $K$ be a field of characteristic 0 and let $R$ be a $K$-algebra essentially of finite type and of diagonally $F$-regular type. Let $d = \dim R$. Then we have $\mathfrak p^{(nd)} \subset \mathfrak p^n$ for all $n$ and all prime ideals $\mathfrak p \subset R$. 
\end{theoremD*}
Thus we see that the affine cone over $\bP^r_\kay \times \bP^s_\kay$ has USTP even if char $\kay = 0$. 

\begin{convention}
Throughout this paper all rings are defined over a field $\kay$ of positive characteristic $p$. Given a ring $R$, we then denote the $e$-th iterate of the Frobenius endomorphism by $F^e\:R \to R$, and use the usual shorthand notation $q \coloneqq p^e$. We assume all rings are essentially of finite type over $\kay$, thus Noetherian, $F$-finite, and so excellent. All tensor products are defined over $\kay$ unless explicitly stated otherwise. We also follow the convention $\mathbb{N}=\{0,1,2,...\}$.
\end{convention}

\subsection*{Acknowledgements} The authors started working on this project during Craig Huneke's 65th Birthday Conference in the University of Michigan in the summer of 2016 after attending Daniel Katz's talk on the subject. We are especially grateful to Craig Huneke who originally inspired us to study this problem for strongly $F$-regular singularities. We are greatly thankful to Elo\'isa Grifo, Linquan Ma, and Ilya Smirnov for many valuable conversations. We are especially thankful to Axel St\"abler, Kevin Tucker, and Robert Walker for reading through a preliminary draft and providing valuable feedback. We are particularly thankful to Axel~St\"abler for pointing out that we do not need $\kay$ to be perfect in \autoref{prop.basicProperties}. Finally, we are deeply thankful to our advisor Karl Schwede for his generous support and constant encouragement throughout this project. 
\section{Preliminaries}
\label{secn.prelim}

The central objects in this paper are Cartier algebras, their test ideals, and the notion of (strong) $F$-regularity of a Cartier module. We briefly summarize these here, following the formalism of M.~Blickle and A.~St\"abler \cite{BlickleTestIdealsViaAlgebras}, \cite{BlickleStablerFunctorialTestModules}. It is worth mentioning that for the most part we will only be using Cartier algebras and test ideals in the generality introduced by K.~Schwede in \cite{SchwedeTestIdealsInNonQGor}.

\begin{definition}[Cartier Algebras]
Let $R$ be a ring. A \emph{Cartier algebra} $\sC$ over $R$ (or \emph{Cartier $R$-algebra}) is an $\mathbb{N}$-graded $\bigoplus_{e\in \bN} \sC_e$  unitary  ring\footnote{Not necessarily commutative.} such that $\sC_0 = R$, and equipped with a graded finitely generated $R$-bimodule structure so that $a \cdot \kappa = \kappa \cdot a^q$, with $\kappa$ homogeneous of degree $e$.\footnote{Recall that if $A$ and $B$ are (commutative) rings, then an $A$-$B$-bimodule is nothing but a left $A \otimes_{\mathbb{Z}} B$-module. More generally, if $A$ and $B$ are algebras over a third ring $R$, then a left $A \otimes_R B$-module is nothing but an $A$-$B$-module, say $M$, with an extra compatibility condition $rm = mr$ for all $r\in R$ and $m\in M$. In this way, all we are saying is that $\sC_e$ is an $R \otimes_{R} F^e_* R $-module.} A \emph{morphism of Cartier algebras} is just a graded homomorphism of unitary rings preserving the $R$-bimodule structures. Note that, strictly speaking, $\sC$ is not an $R$-algebra, as $R$ is not in the center of $\sC$. 
\end{definition}

A central example for us is the \emph{full Cartier algebra} over a ring $R$. This is defined in degree $e>1$ as
\[
\sC_{e,R} \coloneqq \Hom_R (F^e_* R, R).
\]
More generally, given a finite $R$-module $M$ we may define a Cartier algebra $\sC_{M}$ over $R$ as $R$ in degree zero and as
\[
\sC_{e, M} \coloneqq \Hom_R(F^e_* M, M)
\]
in higher degrees. The ring multiplication of $\sC_{M}$ is defined by the rule 
\[\varphi_e \cdot \varphi_d \coloneqq \varphi_e \circ F^e_* \varphi_d \quad \text{for all }\varphi_e \in \sC_{e,M}, \varphi_d \in \sC_{d,M}. \]
Furthermore, the left $R$-module structure of $\sC_{M}$ is the usual one given by post-multiplication, whereas the right $R$-module structure is given by pre-multiplication by elements of $F_*^e R$. More precisely, if $\varphi \in \sC_{e,M}$ and $r \in R$, then
\[
(\varphi \cdot r)(-) =\varphi(F_*^e r \cdot -)
\]

It is worth mentioning we are primarily concerned with Cartier subalgebras of $\sC_R$ in this work. 

\begin{definition}[Cartier Modules]
Given a ring $R$ and a Cartier $R$-algebra $\sC$, we define a \emph{Cartier $\sC$-module} to be a finite $R$-module $M$ equipped with a homomorphism $\sC \to \sC_{M}$ of Cartier $R$-algebras. This is the same as saying $M$ is a left $\sC$-module with coherent underlying $R$-module structure \cite[Lemma 5.2]{BlickleStablerFunctorialTestModules}. A \emph{morphism of Cartier $\sC$-modules} is defined to be a morphism of left $\sC$-modules. 
\end{definition}

Let $R$ be a ring and $\sC$ a Cartier $R$-algebra. Under the assumption $R$ is essentially of finite type over $\kay$, Blickle and St\"abler constructed a covariant functor
\[
\uptau = \uptau(-,\sC) \: \text{left-}\sC\text{-mod} \to R\text{-mod},
\]
from the category of Cartier $\sC$-modules to the category of $R$-modules. This functor is an additive subfunctor of the forgetful functor between these two categories. Thus one has a natural inclusion $\uptau(M, \sC) \subset  M$, where the former module is called the \emph{test submodule of $M$ with respect to $\sC$}, or the \emph{test ideal with respect to $\sC$} in case $M=R$. 


\begin{definition}[$F$-regularity]
A Cartier $\sC$-module $M$ over $R$ is said to be (strongly) $F$-regular (with respect to $\sC$) if the inclusion $\uptau(M, \sC) \subset M$ is an equality.
\end{definition}

Blickle and St\"abler also proved that $\uptau$ commutes with localizations, see \cite[Proposition 1.19.(b)]{BlickleStablerFunctorialTestModules}. In particular, $F$-regularity is a local notion. It follows that $M$ is an $F$-regular Cartier $\sC$-module if and only if $M_{\mathfrak{p}}$ is an $F$-regular Cartier $\sC$-module for all $\mathfrak{p} \in \Spec R$.

On the other hand, suppose $\sC$ is a non-degenerate\footnote{A map $\varphi\colon F^e_*R \to R$ is called nondegenerate if $\varphi(F^e_* R) R_\eta \neq 0$ for all minimal primes $\eta \in \Spec R$. A Cartier algebra $\sC$ is called nondegenerate if $\sC_e$ contains a nondegenerate map for some $e > 0$. See \cite{SchwedeTestIdealsInNonQGor}. } Cartier subalgebra of $\sC_R$, so that $R$ is a Cartier $\sC$-module, and suppose that $R$ is reduced. In this case, as Blickle and St\"abler proved, $\uptau(R, \sC) \subset R$ coincides with the more classically defined (non-finitistic) test ideal of the Cartier algebra $\sC \subset \sC_R$: 
\[
\uptau\big(R, \sC\big) = \{c \in R \mid  \text{for all $r \in R^{\circ}$,  there exists $e$ and $\varphi \in \sC_e$ so that $\varphi(F^e_* r) = c$}\}. \footnote{Any such $c$ is called a \emph{test element for the Cartier algebra}, provided that $c$ is not a zero divisor. The existence of test elements is central to the theory of test ideals.}
\]

Therefore, $R$ is $F$-regular with respect to $\sC \subset \sC_R$ if for every $r \in R^{\circ}$, there exist $e$ and $\varphi \in \sC_e$ such that $\varphi(F^e_* r)=1$. In other words, for all $r \in R^{\circ}$ we have that the $R$-module inclusion
\[
R \to F^e_*R, \quad 1 \mapsto F^e_* r
\]
splits for $e \gg 0$ by a splitting map $\varphi \: F^e_* R \to R$ in $\sC_e$. 

If $\sC \subset \sC_R$, then we say $\sC$ is $F$-regular to mean that $R$ is $F$-regular as a Cartier $\sC$-module. The ring $R$ itself is said to be strongly $F$-regular if $\sC_R$ is $F$-regular.

Finally, if $R$ is reduced and $\sC$ is non-degenerate, the ideal $\uptau\big(R, \sC\big) \subset R$ can be characterized as the smallest ideal of $R$ that contains a non-zero divisor and is compatible with all $\varphi \in \sC_e$, for all $e$. 
In general, an ideal $I \subset R$ is said to be \emph{compatible} with $\varphi \in \Hom_R (F^e_* R, R)$ if 
\[
\varphi(F^e_* I) \subset I.
\]
We may also say either $I$ is $\varphi$-compatible or $\varphi$ is $I$-compatible, see \cite{MehtaRamanathanFrobeniusSplittingAndCohomologyVanishing} and \cite[\S 3A]{SmithZhang}. 

This notion of compatibility between $I$ and $\varphi$ is important because if $\vp$ is $I$-compatible then $\varphi$ restricts to a unique $\overline{\varphi} \in \Hom_{R/I}(F^e_* R/I, R/I)$ making the following diagram commutative
\[
\xymatrix{
F_*^e R \ar[r]^-{\varphi} \ar[d]_-{} & R\ar[d]^-{}\\ 
 F_*^e R/I \ar[r]^-{\overline{\varphi}} & R/I
}
\]

Finally, we record a criterion for verifying the $F$-regularity of a Cartier algebra $\sC \subset \sC_R$ that we will use of later on. We presume it is well-known among experts, but we give a proof for sake of completeness. We are thankful to Karl Schwede for bringing it to our attention, thus significantly simplifying part of our argument.

\begin{proposition}[{\cf\cite[Proposition 4.5]{BMRSClusterAlgebra}, \cite[Theorem 3.3]{HochsterHunekeTightClosureAndStrongFRegularity}, \cite[Lemma 2.3]{BlickleStablerFunctorialTestModules} }] \label{prop.KeyReduction}
Let $R$ be a ring, $\sC \subset \sC_R$ a Cartier $R$-algebra, and $f \in R^\circ$. Suppose that $R_f$ is an $F$-regular $\sC$-module and moreover that there is $\varphi \in \sC_e$, for some $e$, such that $\varphi(F^e_* f) =1$. Then $R$ is an $F$-regular $\sC$-module, i.e. $\sC$ is $F$-regular.
\end{proposition}
\begin{proof}
We must prove that $1 \in \uptau(R, \sC)$. Then our first hypothesis is $\uptau(R_f, \sC) \ni 1$. However, $\uptau(R_f, \sC) = \uptau(R, \sC)_f$. Putting these two statements together we get that $f^n \in \uptau(R, \sC)$ for some $n \in \mathbb{N}$.

If there is $e \in \bN$ and $\psi \in \sC_e$ such that $\psi(F^e_* f^n) = 1$, then we would be done, for
\[
1=\psi(F^e_* f^n) \in \psi\big(F^e_* \uptau(R,\sC)\big) \subset \uptau(R, \sC).
\]

Our second hypothesis is that there exists $e \in \bN$ and $\varphi \in \sC_e$ such that $\varphi(F^e_* f) = 1$. We prove inductively that the same holds for all powers of $f$. Indeed, say $\psi \in \sC_d$ is such that $\psi(F^d_* f^{m-1})=1$. Then it follows that $\vartheta \coloneqq \varphi \cdot \psi \cdot f^{p^d-1}$ is such that $\vartheta(F^{e+d}_*f^m)=1$, for
\begin{align*}
\vartheta(F^{e+d}_* f^m) = \varphi\Big(F_*^e \psi\big(F^d_*f^{p^d-1}f^m\big)\Big) = \varphi\Big(F_*^e \psi\big(F^d_*f^{p^d}f^{m-1}\big)\Big) &= \varphi\Big(F_*^e f\psi\big(F^d_*f^{m-1}\big)\Big)\\
& = \varphi\Big(F_*^e f \cdot 1\Big) =1.
\end{align*}
\end{proof}

\section{Diagonal Cartier algebras and diagonal $F$-regularity}

In \cite{SmolkinSubadditivity}, the second named author introduced the Cartier algebra consisting of $p^{-e}$-linear maps compatible with the diagonal closed embedding $\Delta_2 \: R \otimes R \to R$. Here we generalize this construction to higher diagonals and verify these have the required basic properties, including an analogous subadditivity formula.

For this we consider $\Delta_n \: R^{\otimes n} \to R$ the $n$-th diagonal closed embedding given by the rule $r_1 \otimes \cdots \otimes r_n \mapsto r_1 \cdots r_n$. Recall our convention that all tensor products are defined over $\kay$ unless otherwise explicitly stated. Let $\mathfrak{d}_n$ be the kernel of $\Delta_n$.

\begin{definition}[{Diagonal Cartier algebras, \cf \cite[Notation 3.7]{SmolkinSubadditivity}}] \label{def.diagCartAlg}
Let $R$ be a $\kay$-algebra. For $n\in \bN$ we define the \emph{$n$-th diagonal Cartier algebra of $R/\kay$}, denoted by $\mathcal{D}^{(n)}(R)$, to be in degree $e$
\[
\sD^{(n)}_{e}(R) = \big\{\overline{\varphi} \in \sC_{e,R} \bigm| \text{$\varphi \in \sC_{e,{R^{\otimes n}}}$ and $\varphi$ is $\mathfrak{d}_n$-compatible} \big\}.
\]
In other words, $\sD^{(n)}_{e}(R) \subset \sC_{e,R}$ consists of $R$-linear maps $\varphi \: F^e_* R \to R$ such that there is a lifting $\widehat{\varphi}\: F^e_* R^{\otimes n} \to R^{\otimes n}$ making the following diagram commutative:
\begin{equation*} 
\xymatrix{
F_*^e R^{\otimes n} \ar[r]^-{\widehat{\varphi}} \ar[d]_-{F^e_* \Delta_n} & R^{\otimes n}\ar[d]^-{\Delta_n}\\ 
 F_*^e R\ar[r]^-{\varphi} & R
}
\end{equation*}
It is straightforward to verify $\sD^{(n)}(R)$ is a Cartier subalgebra of $\sC_R$, see \cite[Proposition 3.2]{SmolkinSubadditivity} and \cite[Definition 2.10]{BlickleSchwedeTuckerFSigPairs1}. When the ring $R$ is clear from context, we will refer to this Cartier algebra simply as $\sD^{(n)}$. 
\end{definition}

\begin{definition}[Diagonal $F$-regularity]
We say that a $\kay$-algebra $R$ is \emph{$n$-diagonally $F$-regular} if $\sD^{(n)}(R)$ is $F$-regular. We say that $R$ is \emph{diagonally $F$-regular} if $\sD^{(n)}(R)$ is $F$-regular for all $n\in \mathbb{N}$.
\end{definition}

\begin{remark}[]
Note that $R$ is diagonally $F$-split if and only if $\sD^{(2)}(R)$ is $F$-pure, and so 2-diagonal $F$-regularity can be seen as a strengthening of diagonal $F$-splitting\footnote{See \cite{PayneFrobeniusSplitToric, RamanathanA.1985Svaa, RamananRamanathanProjectiveNormality} for more on diagonal $F$-splittings and why they are important.}.  Indeed, a ring $R$ is defined to be diagonally $F$-split whenever there is a splitting $\varphi \in \sC_{R^{\otimes 2}}$  compatible with $\mathfrak d_2$. It is clear that $\sD^{(2)}(R)$ is $F$-pure whenever $R$ is diagonally $F$-split. On the other hand, suppose that $\varphi \in \sD^{(2)}_e(R)$ is a splitting. Then $\vp$ admits a lifting $\widehat \vp$ in $\sC_{e, R^{\otimes 2}}$, with $\vp(1 \otimes 1) = 1\otimes 1 + f$, for some $f\in \mathfrak d_2$. Further, we have that $\vp \otimes \vp$ is an $F$-splitting of $R^{\otimes 2}$. It follows that $\widehat \vp - f (\vp \otimes \vp)$ is an $F$-splitting of $R^{\otimes 2}$ compatible with $\mathfrak d_2$.
\end{remark}

\subsection{Diagonal test ideals}
The goal in this section is to define the test ideal $\uptau\big(R,\sD^{(n)}; \mathfrak{a}^t\big)$ and record the properties that we will need in the study of USTP, in particular, a subadditivity formula as introduced in \cite{SmolkinSubadditivity}. Of course, this test ideal is nothing but a particular case of $\uptau\big(M,\sC; \mathfrak{a}^t\big)$ for an ideal $\fra$ on $R$ and a nonnegative real number $t$, as in \cite[\S 4]{BlickleStablerFunctorialTestModules}. 

Let $\sC$ be a Cartier $R$-algebra, $\fra \subset R$ an ideal and $t \in \bR_{\geq 0}$. Then one can define a Cartier algebra $\sC^{\fra^t} \subset \sC$ by setting $\sC_e^{\fra^t} \coloneqq \sC_e \cdot \fra^{\lceil t(q-1) \rceil}$ in each degree $e$. Then one defines
\[
\uptau\big(M,\sC; \mathfrak{a}^t\big) \coloneqq \uptau\Big(M,\sC^{\mathfrak{a}^t} \Big) 
\]
The point in making this distinction is mainly ideological. We simply want to think of this object as the test ideal of $\fra^t$ with respect to some extra data. By plugging in $M=R$ and $\sC = \sD^{(n)}$ we get what we call the \emph{$n$-th diagonal test ideal of $\fra^t$}. This test ideal inherits many of the standard properties test ideals enjoy; see \cite[\S 4]{BlickleStablerFunctorialTestModules} for a complete account. However, we isolate the three properties conducive to studying USTP via test/multiplier ideals.

\begin{proposition}[Properties of diagonal test ideals for USTP] \label{prop.basicProperties}
Let $R$ be a reduced $\kay$-algebra, $\sC$ a Cartier $R$-algebra, $\fra\subset R$ an ideal containing a regular element,\footnote{By a regular element we mean a nonzerodivisor.} $t\in \bR_{\geq 0}$ and $n, m \in \bN$. Then the following properties hold.
\begin{enumerate}

\item (Unambiguity) $\uptau\big(R,\sD^{(n)}; \mathfrak{a}^{mt}\big)=\uptau\Big(R,\sD^{(n)}; \big(\fra^m\big)^t\Big)$,

\item (Fundamental lower-bound) $\mathfrak{a} \cdot \uptau\big(R,\sD^{(n)}\big) \subset \uptau\big(R,\sD^{(n)}; \mathfrak{a}\big)$, so that $\mathfrak{a} \subset \uptau\big(R,\sD^{(n)}; \mathfrak{a}\big)$ if $R$ is diagonally $F$-regular,

\item (Subadditivity) $\uptau\big(R, \sD^{(n)}; \fra^{tn}\big) \subset \uptau\big(R, \sC_R; \fra^{t}\big)^n$.
\end{enumerate}

\end{proposition}
\begin{proof} The unambiguity property (a) holds quite generally from observing that
\[
\lceil mt(q-1) \rceil \leq m\lceil t(q-1) \rceil \leq \lceil mt(q-1) \rceil + m
\]
so that
\[
  \fra^{\lceil mt(q-1) \rceil} \supset \fra^{m\lceil t(q-1) \rceil} \supset \fra^{\lceil mt(q-1) \rceil + m} =\fra^{\lceil mt(q-1) \rceil} \cdot \fra^m.
\]
Hence a test element for $\left(\sD^{(n)}\right)^{\fra^{mt}}$ is the same as a test element for $\left(\sD^{(n)}\right)^{(\fra^m)^t}$.

For the fundamental lower-bound (b), if we take $b \in \fra\cap R^\circ$ and $c$ a test element for $\sD^{(n)}$, then $bc$ is a test element for $\left(\sD^{(n)}\right)^{\fra}$. For if $a\in R^{\circ}$, then there exists $\varphi \in \sD^{(n)}_e$ such that $\varphi\big(F^e_* b^{q-1}a\big)=c$. But $\varphi\big(F^e_* b^{q-1}\cdot -\big)$ belongs to $\left(\sD_e^{(n)}\right)^{\fra}$ so we are done.

The proof for the subadditivity formula (c) is similar to the one in \cite{SmolkinSubadditivity}, though here we do not assume that $\kay$ is perfect. As $(F^e_* R)^{\otimes n}$ canonically surjects onto $F^e_* (R^{\otimes n})$, we have 
\[
  \sC_{R^{\otimes n},e} \subset \Hom_{R^{\otimes n}}\bigl( (F^e_*R)^{\otimes n}, R^{\otimes n}\bigr)
\]
for all $e>0$.
Now let $\vp \in\sC_{R^{\otimes n},e}$. By the above inclusion, combined with \cite[Corollary 3.10]{SmolkinSubadditivity}, it follows that $\vp$ induces an element $\vp' \in \Hom_{R}\bigl( F^e_*R, R\bigr)^{\otimes n}$, which we can be expressed as
\[
\vp' = \sum_{j} \vp_{j, 1} \otimes \cdots \otimes \vp_{j,n}. 
\]
where $\vp_{j, k} \in \Hom_R(F^e_* R, R)$. 
Further, given any $x \in \bigl(\fra^{\otimes n} \bigr)^{\lceil t(q-1) \rceil}$, we can write
\[
  x = \sum_i x_{i,1} \otimes \cdots \otimes x_{i, n}
\]
where $x_{i, k} \in \fra^{\lceil t(q-1) \rceil}$. It follows that the map $\vp \cdot x= \vp(F^e_* x \cdot -)$ induces the element
\[
(\varphi \cdot x)'=\vp(F^e_* x \cdot -)' = \sum_{i,j} \vp_{j, 1}(F^e_* x_{i, 1}\cdot - ) \otimes \cdots \otimes \vp_{j,n}(F^e_* x_{i, n}\cdot - ). 
\]
As $\bigl(F^e_* \uptau(R, \sC_R; \mathfrak a^t)\bigr)^{\otimes n}$ canonically surjects onto $F^e_* \bigl(\uptau(R, \sC_R; \mathfrak a^t)^{\otimes n}\bigr)$, we see that
\begin{align*}
 \vp \Bigl(F^e_* \bigl( x \cdot \uptau(R, \sC_R; \mathfrak a^t)^{\otimes n} \bigr)\Bigr) &= \sum_{i,j} \vp_{j, 1}\bigl(F^e_* x_{i, 1} \uptau(R, \sC_R; \mathfrak a^t) \bigr) \otimes \cdots \otimes \vp_{j,n}\bigl(F^e_* x_{i, n} \uptau(R, \sC_R; \mathfrak a^t) \bigr)\\
 &\subset  \uptau(R, \sC_R; \mathfrak a^t)^{\otimes n}
\end{align*}
Thus we obtain
\[
  \uptau\left( R^{\otimes n}, \sC_{R^{\otimes n}}; \left(\mathfrak a^{\otimes n}\right)^t  \right) \subset \left(  \uptau(R, \sC_R; \mathfrak a^t) \right)^{\otimes n}
\]
by the minimality of the test ideal on the left. Then we apply $\Delta_n$ to both sides. On the right-hand side we get $\uptau(R, \sC_R; \mathfrak a^t)^n$. On the left-hand side we get 
something larger than $\uptau(R, \sD^{(n)}; \mathfrak a^{tn})$, by \cite[Proposition 3.6]{SmolkinSubadditivity}. The fact that $\sD^{(n)}$ is non-degenerate follows \emph{mutatis mutandis} from the same argument as in \cite[Theorem 3.11]{SmolkinSubadditivity}.
\end{proof}

We  develop a broader theory of diagonal Cartier algebras and diagonal $F$-regularity in a forthcoming preprint \cite{CarvajalSmolkinPristineMorphisms}. 

\section{USTP for diagonally $F$-regular singularities}

In this section we prove our main result, namely that USTP is satisfied by locally diagonally $F$-regular rings with $h$ equal to the dimension. We do this by making our discussion in the introduction rigorous. For this we establish:

\begin{theorem} \label{thm.MainTheorem}
Let $R$ be a diagonally $F$-regular $\kay$-algebra, and let $\mathfrak p \in \Spec R$ be an ideal of height $h$. Then $ \mathfrak{p}^{(hn)} \subset \mathfrak{p}^n $
for all $n \in  \mathbb{N}$. \label{thm.diagFRegUSTP}
\end{theorem}


\begin{proof} \label{rem.RemarkGemini}  This containment of ideals can be checked locally, and so we may assume that $R$ is local. We can also assume that $\mathfrak p$ is not the maximal ideal of $R$, because in that case $\mathfrak{p}^{(n)} = \mathfrak p^{n}$ for all $n$. This implies that the residue field of $R$ at $\mathfrak p$ is transcendental over $\kay$, and so $\kappa(\mathfrak p)$ is infinite.\footnote{Recall that $\kappa(\mathfrak p)/\kay$ is algebraic if and only if $\mathfrak p$ is maximal in $R$. }

As mentioned in the introduction, our strategy for proving this theorem is to enlarge the scope of the proof in \cite[\S 6.3]{SchwedeTuckerTestIdealSurvey} and \cite[\S 4.3]{SmithZhang}. We just need to verify that the upper-bound
\begin{equation} \label{eqn.UpperBound}
\uptau\bigg(R, \sC_R;\Big(\mathfrak{p}^{(hn)} \Big)^{1/n} \bigg) \subset \mathfrak{p}
\end{equation}
holds for all $n\in \bN$, all prime ideals $\mathfrak{p} \subset R$, and all $R$ under our consideration. This inclusion can be checked after localizing at $\mathfrak{p}$, which means that we may assume $R$ is local with maximal ideal $\mathfrak{p}$ and infinite residue field. But then in that case $\mathfrak{p}^{(hn)}=\mathfrak{p}^{hn}$. Therefore, the left-hand side in \autoref{eqn.UpperBound} simply becomes
\[
\uptau\bigg(\Big(\mathfrak{p}^{(hn)} \Big)^{1/n} \bigg)=\uptau\Big(\big(\mathfrak{p}^{hn} \big)^{1/n} \Big) = \uptau \Big( \mathfrak{p}^{hn/n} \Big) = \uptau\big(\mathfrak{p}^h\big).
\]
Using \cite[Theorems 8.3.7 and 8.3.9]{HunekeSwansonIntegralClosure}, just as in \cite[Proof of Theorem 6.23]{SchwedeTuckerTestIdealSurvey}, we have that $\mathfrak{p}$ admits a reduction,\footnote{That is, a subideal with the same integral closure.} say $\mathfrak{q} \subset \mathfrak{p}$, generated by less than $h=\dim R_{\mathfrak{p}}$ elements.\footnote{Here it is where we need the residue field to be infinite.} Hence,
\[
\uptau\big(\mathfrak{p}^h\big)=\uptau\big(\mathfrak{q}^h\big) \subset \mathfrak{q} \subset \mathfrak{p},
\]
where the penultimate inclusion is nothing but a consequence of the Brian\c con--Skoda theorem for test ideals \cite{HochsterHunekeTightClosureInvariantThoeryBrianconSkoda}, \cite[Proposition 4.2]{BlickleStablerFunctorialTestModules}. The equality simply follows from unambiguity and the invariance of test ideals under integral closure, see \cite[Theorem 6.9]{SchwedeTuckerTestIdealSurvey}.

Thus, for all $\mathfrak{p} \in \Spec R$ and $n \in \mathbb{N}$ we have the following:
\begin{equation*}
  \mathfrak p^{(hn)} \overset{(1)}{\subset}
  \uptau \Big(R,\sD^{(n)};\mathfrak p^{(hn)}\Big) \overset{(\textrm{\Bicycle})}{=} 
  \uptau\left( R, \sD^{(n)}; \Big(\mathfrak p^{(hn)}\Big)^{n/n} \right) \overset{(2)}{\subset}
  \uptau\left(R, \sC_R; \Big(\mathfrak p^{(hn)}\Big)^{1/n} \right)^n \overset{(3)}{\subset}
  \mathfrak p^n
\end{equation*}
Here, (1) follows from $R$ being diagonally $F$-regular and \autoref{prop.basicProperties}. The equality (\Bicycle) is simply unambiguity, whereas (2) follows from subadditivity and (3) is just \autoref{eqn.UpperBound} raised to the $n$-th power.
\end{proof}
\begin{remark}
  Thus, if $R$ is diagonally $F$-regular, we have $\mathfrak p^{(dn)} \subset \mathfrak p^n$, where $d = \dim R$, for \emph{all} $\mathfrak p \in \Spec R$. If $R$ is local or graded, then
  in fact $\mathfrak p^{((d-1)n)} \subset \mathfrak p^n$ holds because symbolic and ordinary powers of the maximal ideal are the same. 
\end{remark} 

\section{On the class of diagonally $F$-regular rings}

Here is a simple observation about the class of diagonally $F$-regular rings.

\begin{proposition}
Essentially smooth $\kay$-algebras are diagonally $F$-regular. Further, $n$-diagonally $F$-regular $\kay$-algebras are strongly $F$-regular, in particular normal and Cohen--Macaulay.
\end{proposition}
\begin{proof}
The second statement is obvious, whereas the former is a consequence of Kunz's theorem \cite{KunzCharacterizationsOfRegularLocalRings} just as in \cite[\S 7]{SmolkinSubadditivity}. Indeed, if $R$ is smooth over $\kay$, then $R^{\otimes n}$ is smooth and therefore regular for all $n$. Thus Kunz's theorem tells us that $F^e_* R^{\otimes n}$ is a projective $R^{\otimes n}$-module, which implies that $\sD^{(n)}(R) = \sC_R$ for all $n$. Similarly, if $R$ is a localization of $S$, where $S$ is a smooth $\kay$ algebra, then $R^{\otimes n}$ is a localization of $S^{\otimes n}$ and so $R^{\otimes n}$ is still regular. The result follows.
\end{proof}

It follows from the following proposition that the class of diagonally $F$-regular $\kay$-algebras is properly contained in the class of strongly $F$-regular ones. In the next subsection, we will show that the class of diagonally $F$-regular $\kay$-algebras  properly contains the class of essentially smooth algebras. 

We thank Linquan~Ma for giving us the following observation:

\begin{proposition}
Let $(R, \mathfrak m)$ be a local normal domain essentially of finite type over $\kay$, with $R/\mathfrak m$ infinite. If $R$ is diagonally $F$-regular, then the divisor class group $\Cl(R)$ is torsion-free. In fact, if $\Cl(R)$ has $r$-torsion\footnote{That is, some element of $\Cl(R)$ is annihilated by $r$.}, then $\sD^{(nr)}(R)$ is not $F$-regular for $n \gg 0$. \label{prop.diagFRegIsTorsFree}
\end{proposition}
\begin{proof}
Suppose $\sD^{(nr)}$ is $F$-regular for all $n$. Then for all prime ideals $\mathfrak p$, we have 
\[
  \mathfrak p^{(hnr)} \subset \mathfrak p^{nr}
\]
where $h$ is the height of $\mathfrak p$. By assumption, there exists some non-principal prime ideal $\mathfrak q$ in $R$ of height 1 such that $\mathfrak q^{(r)}$ is principal. Thus $\mathfrak q^{(rn)} = \mathfrak q^{rn}$ is principal for all $n$. 

However, this cannot happen. Indeed, since $R$ is a normal domain, we know that principal ideals of $R$ are integrally closed, and so the analytic spread of $\mathfrak q$ is at least $2$. This tells us that the fiber cone of $\mathfrak q$, 
\[
  F_{\mathfrak q} = \frac{R}{\mathfrak m} \oplus \frac{\mathfrak q}{\mathfrak m \mathfrak q} \oplus \frac{\mathfrak q^2}{\mathfrak m \mathfrak q^2}  \oplus \cdots
\]
has dimension at least 2, so the Hilbert function of $F_{\mathfrak q}$, $h(F_{\mathfrak q}, n)$, agrees with a non-constant polynomial for $n \gg 0$. But we know that $h(F_{\mathfrak q}, n) = \mu(\mathfrak q^n)$ by Nakayama's lemma, so $\mathfrak q^{rn}$ is not principal for $n \gg 0$. 
\end{proof}

\begin{example}
By the above proposition, we see that Veronese subrings of polynomial rings are never diagonally $F$-regular, \cf \cite[Example 6.9]{SmolkinSubadditivity}.
\end{example}

By \cite[Theorem G]{CarvajalFiniteTorsors}, if $s(R) > 1/2$ then $\Cl(R)$ is torsion-free. In light of \autoref{prop.diagFRegIsTorsFree}, we suspect there is an interesting connection between diagonally $F$-regular rings and rings with $F$-signature greater than $1/2$. For example, we pose the following question:

\begin{question}
  If $s(R) > 1/2$, is $R$ diagonally $F$-regular? In particular, is 
  \[
  \kay[x_1, \ldots, x_d]\bigm/\bigl(x_1^2 + \cdots + x_d^2\bigr)
 \]
diagonally $F$-regular for all $d \geq 4$? We note that there exist diagonally $F$-regular rings with $F$-signature less than $1/2$, by \autoref{thm.SegreDFR} and work of A.~Singh \cite[Example 7]{SinghFSignatureOfAffineSemigroup}.
\end{question}

The following proposition shows that the class of diagonally $F$-regular $\kay$-algebras is closed under tensor product. 
\begin{proposition}\label{prop.Products} Let $R$ and $S$ be $n$-diagonally $F$-regular $\kay$-algebras. Then $R\otimes S$ is a $n$-diagonally $F$-regular $\kay$-algebra.
  \label{prop.prodIsDiagFReg}
\end{proposition}
\begin{proof}
  We prove this via global $F$-signatures \cite{deStefaniPolstraYao-Globalizing}. For simplicity, write $a  = a_e\bigl(R, \sD^{(n)}\bigr)$ and $b = a_e\bigl(S, \sD^{(n)}\bigr)$. Suppose 
  \[
    \varphi : F^e_* R \twoheadrightarrow R^{\oplus a}, \quad \psi: F^e_* S \twoheadrightarrow S^{\oplus b}
  \] 
  are surjections, such that each composition
  \[
  F^e_* R \xrightarrow{\varphi} R^{\oplus a} \xrightarrow{\pi_i} R
  \]
  is in $\sD^{(n)}(R)$. Similarly, each composition 
  \[
  F^e_* S \xrightarrow{\psi} S^{\oplus b} \xrightarrow{\sigma_j} S
  \]
  is in $\sD^{(n)}(S)$. Then we get a surjection of $R\otimes S$-modules
  \[
    F^e_* (R \otimes S) \cong F^e_* R \otimes F^e_* S \xrightarrow{\varphi \otimes \psi} R^{\oplus a} \otimes  S^{\oplus b} \cong \left( R \otimes S \right)^{\oplus ab}
  \]
  Then each composition 
  \[
    (\pi_i \circ \varphi) \otimes (\sigma_j \circ \psi) \colon F^e_* (R \otimes S) \xrightarrow{\varphi \otimes \psi} \left( R \otimes S \right)^{\oplus ab} \xrightarrow{\pi_i \otimes \sigma_j} R \otimes S
  \]
  is in the Cartier algebra $\sD^{(n)}(R \otimes S)$. Indeed, given any maps $\theta \in \sD^{(n)}_e(R)$ and  $\eta \in \sD^{(n)}_e(S)$, with liftings $\widehat \theta \in \sC_{e,R^{\otimes n}}$ and $\widehat \eta \in \sC_{e, S^{\otimes n}}$, one checks that $\widehat \theta \otimes \widehat \eta$ is a lifting of $\theta \otimes \eta$ by a diagram chase. Thus, $a_e\bigl(R \otimes S , \sD^{(n)}(R\otimes S)\bigr) \geq ab$. It follows that 
  \[
  s\bigl(R \otimes S, \sD^{(n)}(R \otimes S)\bigr) \geq s\bigl(R, \sD^{(n)}(R)) \cdot s(S, \sD^{(n)}(S)\bigr) > 0,
  \]
  as desired. 
\end{proof}

\subsection{Segre products of polynomial rings are diagonally $F$-regular} \label{section:segre}
The remainder of this section will be spent proving the following theorem
\begin{theorem} \label{thm.SegreDFR}
 Let $R$ be the Segre product $\kay[x_0, \ldots , x_r]\# \kay[y_0, \ldots, y_s]$, with $r,s > 0$, and $\kay$ perfect. Then $R$ is diagonally $F$-regular.
\label{thm.segreDiagFreg}
\end{theorem}

Combined with \autoref{thm.diagFRegUSTP}, we get the following corollary: 
\begin{corollary}
  Let $R = \kay[x_0, \ldots , x_r]\# \kay[y_0, \ldots, y_s]$, and let $\mathfrak{p} \subset R$ be a prime ideal. Then $\mathfrak p^{(hn)} \subset \mathfrak p^n$ for all $n$, where $h = \dim R-1 = r+s$. 
\end{corollary}

\begin{remark}
Let $\ell/\kay$ be a finitely generated field extension over a perfect field. Then
in view of \autoref{prop.Products} and \autoref{thm.SegreDFR} we have that $R_{\ell}=\ell[x_0, \ldots , x_r]\# \ell[y_0, \ldots, y_s]$ is a diagonally $F$-regular $\kay$-algebra. In particular, USTP holds for  $R_{\ell}$ as well.
\end{remark}

Combining \autoref{thm.SegreDFR} and \autoref{prop.Products}, we obtain the following observation:
\begin{corollary}
The class of diagonally $F$-regular $\kay$-algebras includes some non-isolated singularities.
\end{corollary}

We now prove \autoref{thm.SegreDFR}. 
Observe that $R$ can be realized as the following subring of $S:=\kay[x_0, \ldots, x_r, y_0, \ldots, y_s]$:
\[
R=\kay[x_0 y_0, \ldots, x_i y_j , \ldots, x_r y_s] \subset \kay[x_0, \ldots, x_r, y_0, \ldots, y_s]=S.
\]
Fix an integer $n > 1$. We wish to show that, for all $f\in R$, there exist $e \geq 0$ and $\varphi \in \diagCA{n}_e$ such that $\varphi(F^e_* f) = 1$. We have the following lemma: 

\begin{lemma}
\label{lemma:splitAJacobianElement} Let $A$ be a $\kay$-algebra, where $\kay$ is perfect. Let $f$ be an element of $A^\circ$ such that $A_f$ is regular. Suppose also that there exist $e>0$ and $\psi \in \sD^{(n)}_e(A)$ with $\psi(F^e_* f)  =1 $. Then $A_f$ is an $F$-regular $\sD^{(n)}(A)$-module for all $n > 0$. 
\end{lemma}
\begin{proof}
We want to show that $\uptau \coloneqq \uptau\bigl(A_f, \sD^{(n)}(A)\bigr) = A_f$. \emph{A priori}, $\uptau$ is an $A$-submodule of $A_f$. However, by \cite[Proposition 1.19b]{BlickleStablerFunctorialTestModules}, we know that $\uptau$ is an ideal of $A_f$. 
\begin{claim}
Let $x$ be an arbitrary element of $A_f^\circ$. Then there exist $e' > 0$ and $\vp \in \sD^{(n)}_{e'}(A)$ such that $\vp \cdot x =(\vp/1)\bigl(F^{e'}_*x\bigr) = 1$, where we used the canonical isomorphism $\Hom_{A_f}\bigl(F^{e'}_* A_f, A_f\bigr) = \Hom_A\bigl(F^{e'}_* A, A\bigr)_f$ to realize the action of $\vp$ on $x$. 
\end{claim}
\begin{proof}[Proof of claim]
As $A_f$ is regular and $\kay$ is perfect, we know that $A_f$ is diagonally $F$-regular, and so there exists $\phi\in \sD^{(n)}(A_f)$ with $\phi\bigl(F^{e'}_* x\bigr) = 1$. Further, as $\Hom_{A_f}\bigl(F^{e'}_* A_f, A_f\bigr) = \Hom_A\bigl(F^{e'}_* A, A\bigr)_f$, we can write $\phi = \vartheta/f^j$ for some $j$, where $\vartheta \in \Hom_A\bigl(F^{e'}_* A, A\bigr)$. It follows that there exists $i$ such that  $f^i \vartheta \in \sD^{(n)}(A)$. Now we have\footnote{Note that, \emph{a priori}, we only have this equation after multiplying both sides by a sufficiently large power of $f$. However, we get this equation by virtue of $f$ being a nonzerodivisor.}
\[
  f^i \vartheta \bigl( F^{e'}_* x \bigr) = f^{i+j}.
\]
By hypothesis, there exist $e>0$ and $\psi \in \sD^{(n)}_e(A)$ with $\psi(F^e_* f) = 1$. As in the proof of \autoref{prop.KeyReduction}, there exist $e''>0$ and $\psi'' \in \sD^{(n)}_{e''}(A)$ with $\psi''(F^{e''}_* f^{i+j}) = 1$. Then we get the desired map by taking $\varphi = \psi'' \cdot f^i \vartheta$. This proves the claim.
\end{proof}

By the claim, $A_f$ is an $F$-pure $\sD^{(n)}(A)$-module. By definition, we have that $H^0_\eta(\uptau_\eta) \subset H^0_\eta\bigl((A_f)_\eta\bigr)$ is a nil-isomorphism\footnote{See \cite[\S 1]{BlickleStablerFunctorialTestModules} for the definition of a nil-isomorphism.} for every associated prime $\eta \in \operatorname{Ass}_A(A_f)$, where $H^0_\eta$ denotes the local cohomology functor. This means $\uptau$ contains a nonzerodivisor of $A_f$. Indeed, if this is not the case, then $\uptau \subset \bigcup_{\eta \in \operatorname{Ass}(A_f)} \eta$, and so $\uptau \subset \eta$ for some $\eta \in \operatorname{Ass}(A_f)$ by prime avoidance. Further, we know that $\eta = \eta' A_f$ for some $\eta' \in \operatorname{Ass}_A(A_f)$. It follows that $H^0_{\eta'} (\uptau_{\eta'}) = H^0_\eta(\uptau_\eta) = \uptau_\eta$, as $\eta$ is a nilpotent ideal in $A_\eta$. Similarly, $H^0_{\eta'}\bigl((A_f)_{\eta'}\bigr) = H^0_\eta(A_\eta) = A_\eta$. As $A_f$ is $F$-pure as a $\sD^{(n)}(A)$-module, so is $A_\eta$. It follows that  $\bigl(\sD^{(n)}(A)\bigr)_+^N A_\eta = A_\eta$ for all $N> 0$, so the inclusion $H^0_{\eta'}(\uptau_{\eta'}) \subset H^0_{\eta'}\bigl((A_f)_{\eta'}\bigr)$ is not a nil-isomorphism. 

As $\uptau$ contains a nonzerodivisor, and $\uptau$ is a $\sD^{(n)}(A)$-submodule of $A_f$, it follows from the claim that $1 \in \uptau$. As $\uptau$ is an ideal of $A_f$, it follows that $\uptau = A_f$, as desired. 
\end{proof}

By \autoref{prop.KeyReduction} combined with the above lemma, to prove \autoref{thm.SegreDFR} it suffices to find an integer $e$ and a map $\varphi \in \diagCA{n}_e(R)$ with $\varphi(F^e_* x_0 y_0)=1$. It turns out that finding the correct map $\varphi$ is easy; the hard part is checking that $\varphi \in \diagCA{n}(R)$. Our strategy will be to work mostly in the polynomial ring $S$. This is possible thanks to the following lemma:
\begin{lemma} The Frobenius trace $\Phi^e \in \sC_{e,S}$ restricts to a map in $\sC_{e,R}$, \ie  $\Phi^e(F^e_* R) \subset R$, so that there is a commutative diagram
\[
\begin{tikzcd}
  F^e_*R \arrow{r}{\Phi^e} \arrow[hook]{d} & R \arrow[hook]{d}\\
  F^e_*S \arrow{r}{\Phi^e}  & S 
\end{tikzcd}
\]
\label{lemma.Phirestricts}
\end{lemma}
\begin{proof}
Let $x_0^{a_0} \cdots x_r^{a_r} \cdot y_0^{b_0} \cdots y_s^{b_s}$ be a monomial in $R$, meaning 
\begin{equation} \label{eqn.MonomialR}
 \sum_{i=0}^r a_i =  \sum_{i=0}^s b_i.
\end{equation}
For convenience, we will use the notation 
\[
  \exes a \coloneqq x_0^{a_0} \cdots x_r^{a_r}, \quad \whys b  \coloneqq y_0^{b_0} \cdots y_s^{b_s}.
\]
Write using the Eucliden algorithm,
\begin{equation} \label{eqn.DecompostionA}
a_i \eqqcolon \mu_i q + \alpha_i, \quad {0 \leq \alpha_i \leq q-1}.
\end{equation}
Similarly,
\begin{equation} \label{eqn.DecompostionB}
b_i \eqqcolon \nu_i q + \beta_i, \quad {0 \leq \beta_i \leq q-1}.
\end{equation}
In such a way that,
\[
F^e_* \exes{a} \whys{b} = \exes{\smash{\mu}} \whys{\nu} \cdot F^e_* \exes{\alpha} \whys{ \smash \beta}.
\]
Therefore,
\[
\Phi^e \big( F^e_* \exes{a} \whys{b} \big) = \begin{cases}
\exes{\smash \mu} \whys{\nu}  & \text{if } \alpha_i, \beta_i = q-1,\\
\hfil 0 & \text{otherwise}.
\end{cases}
\]
Now, combining \autoref{eqn.MonomialR}, \autoref{eqn.DecompostionA} and \autoref{eqn.DecompostionB} we get
\begin{equation} \label{eqn.KeyEquation}
\left( \sum \mu_i \right) q + \sum \alpha_i = \left( \sum \nu_i\right) q + \sum \beta_i,
\end{equation}
Introducing the notation $\mu \coloneqq \sum_i \mu_i$ etcetera, we conclude that $\mu= \nu$ if and only if $\alpha = \beta$. In particular, if $\alpha_i ,\beta_i = q-1$ then $\mu = \nu$, meaning that
\[
 \exes{\smash \mu} \whys{\nu}  \in R,
\]
as desired. This proves the lemma.
\end{proof}

Continuing with the proof of \autoref{thm.SegreDFR}, consider the map 
\[\varphi_e:=\Phi^e \cdot x_0^{q-2} x_1^{q-1} \cdots x_r^{q-1} y_0^{q-2} y_1^{q-1} \cdots y_s^{q-1} \in \sC_e^S .\] 
Since 
\[
x_0^{q-2} x_1^{q-1} \cdots x_r^{q-1} y_0^{q-2} y_1^{q-1} \cdots y_s^{q-1} \in R,
\]
we have that $\varphi_e$ also restricts to a map in $\sC_e^R$. Moreover,
\[
\varphi_e(F^e_* x_0 y_0) = \Phi^e \Big(F^e_* x_0^{q-1} \cdots x_r^{q-1} y_0^{q-1} \cdots y_s^{q-1}\Big) = 1.
\]
Hence, it suffices to prove that $\varphi_e \in \sD^{(n)}(R)$  for $e$ large enough. Our strategy will be to show the following.
\begin{claim} \label{claim.main}
There exists a lifting of $\varphi_e \in \sC_{e,S}$ to $\sC_{e, S^{\otimes n}}$, say
\[
\xymatrix{
F_*^e S^{\otimes n} \ar[r]^-{\widehat{\varphi}_e} \ar[d]_-{F^e_* \Delta_n} & S^{\otimes n}\ar[d]^-{\Delta_n}\\ 
 F_*^e S \ar[r]^-{\varphi_e} & S
}
\]
such that $\widehat{\varphi}_e$ restricts to $R^{\otimes n}$, \ie $\widehat{\varphi}_e \big(R^{\otimes n}\big) \subset R^{\otimes n}$, for $e \gg 0$.
\end{claim}
It suffices to show this claim, for then the restriction of $\widehat{\varphi}_e$ to $F^e_* R^{\otimes n}$ will be a lifting of $\varphi_e \: F^e_* R \to R$. We are going to spend the rest of the section proving \autoref{claim.main}. For this, we use the following notation,
\[
S^{\otimes n}=\kay[ \bm x_1, \bm y_1, \ldots, \bm x_n, \bm y_n],
\]
where
\[ 
  \bm x_k \coloneqq x_{0, k}, x_{1,k}, \ldots, x_{r, k}
\]
and similarly for $\bm y_k$, where the second subscript of $x_{i,k}$ (resp. $y_{j,k}$) denotes which copy of the $n$-fold tensor product it corresponds to. We also write
\[
R^{\otimes n}=\kay \left[ x_{i, k} y_{j, k} \, \Bigg| %
\begin{array}{l} %
1 \leq i \leq r,\\ %
0 \leq j \leq s,\\ %
0 \leq k \leq n %
\end{array} %
 \right]
\]
so that a monomial
\[
\prod_{k=1}^n \exesi{a}{k} \whysi{b}{k} \in S^{\otimes n}
\]
belongs to $R^{\otimes n}$ if and only if 
\[
 a_k = b_k
\]
for all $k$, where we use the notation
\[
  a_k \coloneqq \sum_{i=0}^r a_{i, k}
\text{\quad and \quad}
  \exesi{a}{k} \coloneqq \prod_{i=0}^r x_{i, k}^{a_{i, k}}
\]
and similarly for $b_k$ and $\whysi{b}{k}$. To be clear, the second subscript always denotes which factor of the $n$-fold tensor product we are working in. 

Recall that
\[
F^e_*\prod_{k=1}^n \exesi{a}{k} \whysi{b}{k}, \quad 0 \leq a_{i,k}, b_{j,k} \leq  q-1
\]
is a (free) basis of $F_*^e S^{\otimes n}$ as an $S^{\otimes n}$-module.  We will construct the map $\widehat{\varphi}_e$ from \autoref{claim.main} explicitly by assigning values for $\widehat{\varphi}_e$ at each of these basis elements, pursuant to the two conditions:
\begin{enumerate}
\item $\Delta_n \circ \widehat{\varphi}_e = \varphi_e \circ F^e_* \Delta_n$, and 
\item $\widehat{\varphi}_e \big(F^e_* R^{\otimes n}\big) \subset R^{\otimes n}$.
\end{enumerate}
For some basis elements, it is easy to figure out where we can send them. For others, it is a more delicate question. We begin by taking care of the easy ones. 

As we will see, one case when it is easy is when our basis element is in the kernel of $\varphi_e \circ F^e_* \Delta_n$. Let $\psi:= \varphi_e \circ F^e_* \Delta_n$. Then we have
\begin{align*}
&\psi \Bigg( F^e_*\prod_{k=1}^n \exesi{a}{k} \whysi{b}{k} \Bigg) \\
=& \varphi_e \Big( F^e_* \exestothe{\sum_k a_{\bullet, k}} \whystothe{\sum_k b_{\bullet, k}} \Big)\\
=& \Phi^e \Big( F^e_* x_0^{q-2+\sum_k a_{0, k}}x_1^{q-1+\sum_k a_{1, k}} \cdots x_r^{q-1+\sum_k a_{r, k}} \cdot y_0^{q-2+\sum_k b_{0, k}} y_1^{q-1+\sum_k b_{1, k}} \cdots y_s^{q-1+\sum_k b_{s, k}}  \Big).
\end{align*}
This will be nonzero precisely when 
\begin{align} \label{eqn.Lala} \tag{\Cat}
\begin{split}
\sumofas{0}, \sumofbs{0} & \equiv 1 \mod q, \\
\sumofas{i}, \sumofbs{j} &\equiv 0 \mod q, \textrm{ where } 1 \leq i \leq r, 1 \leq j \leq s. 
\end{split}
\end{align}
Let $\upsilon(x) = \floor{x/q}$. Hence, in case \autoref{eqn.Lala} we have
\begin{align*}
&\Phi^e \Big( F^e_* x_0^{q-2+\sumofas{0}}x_1^{q-1+\sumofas{1}} \cdots x_{r\phantom{1}}^{q-1+\sum_k a_{r, k}} \cdot y_0^{q-2+\sum_k b_{0, k}} y_1^{q-1+\sum_k b_{1, k}} \cdots y_{s\phantom{1}}^{q-1+\sum_k b_{s, k}}  \Big)\\ = &\exestothe{\upsilon\left(\sumofas{\bullet} \right)} \whystothe{\upsilon\left(\sumofbs{\bullet} \right)}. 
\end{align*}
In summary,
\begin{equation*}
\psi \left( F^e_*\prod_{k=1}^n \exesi{a}{k} \whysi{b}{k} \right)
= \begin{cases}
\exestothe{\upsilon\left(\sumofas{\bullet} \right)} \whystothe{\upsilon\left(\sumofbs{\bullet} \right)}  & \textrm{if  condition \autoref{eqn.Lala} holds,}\\
\hfil 0 & \text{otherwise.}
\end{cases}
\end{equation*}
If condition \autoref{eqn.Lala} does not hold, we set 
\[
\widehat{\varphi}_e \left( F^e_*\prod_{k=1}^n \exesi{a}{k} \whysi{b}{k} \right) =0 \in R^{\otimes n}.
\]

The next case that is easy to deal with is the case where our generator of $F^e_* S^{\otimes n}$ has nothing to do with $F^e_* R^{\otimes n}$. More precisely, if we have
\[
 \left( S^{\otimes n} F^e_* \prod_{k=1}^n \exesi{a}{k} \whysi{b}{k} \right) \cap F^e_* R^{\otimes n} = 0
\]
then the value we assign to
\[
\widehat{\varphi}_e \left( F^e_*\prod_{k=1}^n \exesi{a}{k} \whysi{b}{k} \right)
\]
has no bearing on whether $\widehat{\varphi}_e(F^e_* R^{\otimes n}) \subset R^{\otimes n}$. So for these generators we only need to worry about the requirement that $\Delta_n \circ \widehat{\varphi}_e = \varphi_e \circ F^e_* \Delta_n$. We deduce which generators satisfy this condition in the following lemma. 

\begin{lemma} \label{rem.generators}
 $F^e_*R$ is generated as an $R$-submodule of $F^e_* S$ by the  elements
\begin{equation*}
\exes{\mu} \cdot F^e_* \exes{\alpha} \whys{\beta}, \quad{0\leq \alpha_i, \beta_j \leq q-1}
\end{equation*}
such that $sq \geq \mu q = \beta - \alpha \geq 0$, along with the elements 
\begin{equation*}
\whys{\nu} \cdot F^e_* \exes{\alpha} \whys{\beta}, \quad{0\leq \alpha_i, \beta_j \leq q-1}
\end{equation*}
such that $rq \geq \nu q = \alpha - \beta \geq 0$.
Moreover, $F^e_* R^{\otimes n}$ is generated as an $R^{\otimes n}$-module by tensor products of these generators. Here, we are still using the notation $\mu = \sum_{i=0}^r \mu_i$ and $\nu = \sum_{j=0}^s \nu_j$, and similarly for $\alpha$ and $\beta$.  

In particular, the ring $F^e_* R^{\otimes n}$ is contained in the direct summand of $F^e_* S^{\otimes n}$ generated  as a (free) $S^{\otimes n}$-module by monomials of the form
\[
  F^e_* \prod_k \exesi{a}{k} \whysi{b}{k}
\]
such that $ b_k - a_k \equiv 0 \mod q$ for all $k$, $1 \leq k \leq n$. Here, we are still using the notation $b_k \coloneqq \sum_{j=0}^s b_{j, k}$ and $a_k \coloneqq \sum_{i=0}^r a_{i, k}$.
\end{lemma}
\begin{proof}
We observed in the proof of \autoref{lemma.Phirestricts} that elements in $F_*^e R$ are $\kay$-linear combinations of elements of the form
\[
\exes{\mu} \whys{\nu} \cdot F^e_* \exes{\alpha} \whys{\beta}, \quad 0\leq \alpha_i, \beta_j \leq q-1
\]
such that
\[
\mu q + \alpha = \nu q + \beta,
\]
 equivalently,
 \begin{equation} \label{eqn.KeyEqnAgain}
 (\mu-\nu)q=\beta - \alpha.
 \end{equation}
In particular, $\mu - \nu$ and $\beta - \alpha$ have both the same sign (including zero). Note that 
\[
  \beta - \alpha \in \{-(r+1)(q-1), - (r+1)(q-1)+1, \ldots ,-1, 0, 1, \ldots, (s+1)(q-1)\}
  \]
  and $(\mu-\nu)q \in q \bZ $. We see that the intersection of these two sets is $\{-rq, -rq+1, \ldots, sq\}$, assuming $q > \max\{r+1, s+1\}$. Therefore for \autoref{eqn.KeyEqnAgain} to hold there are three possibilities: if $\mu - \nu = 0$, then both monomials $\exes{\mu} \whys{\nu}$ and $\exes{\alpha} \whys{\beta}$ are in $R$. Otherwise, if $\mu - \nu >0 $ (respectively, $\mu - \nu < 0$), then the monomial $\exes{\mu} \whys{\nu}$ can be factored as a product of a monomial in $R$ times a monomial $\exes{\mu'}$ (respectively,  $\whys{\nu'} $) with $\mu' = \mu - \nu$ (respectively, $\nu' = \nu - \mu$). 
This proves the lemma.
\end{proof}

The above being said, we proceed as follows. If we have  $ b_k - a_k \not \equiv 0 \mod q$ for some $1 \leq k \leq n$, we set
\[
 \widehat{\varphi}_e \left( F^e_* \prod_k \exesi{a}{k} \whysi{b}{k}\right) =\psi \Big( F^e_* \prod_k \exesi{a}{k} \whysi{b}{k}\Big) \otimes 1 \otimes \cdots \otimes 1 .
\]
Note that this is consistent with our earlier assignment, even if condition \autoref{eqn.Lala} does not hold.

Now we come to the hard part of this proof. We are given a monomial that satisfies condition \autoref{eqn.Lala} and also satisfies $b_k - a_k \equiv 0 \mod q$ for all $k$ and we need to figure out where $\widehat \vp_e$ should send it to. Our idea is quite simple, though it might be lost in the cumbersome notation. Thus it makes sense to do an example first. 

\begin{example} 
Say $p=5, e=1, n = 2,$ and $r = s = 1$. Let $F_* g \coloneqq F_* x_{0, 1} x_{1, 1} y_{0, 1}^{3} y_{1, 1}^4 \otimes x_{1, 2}^4 y_{0, 2}^3 y_{1, 2}$ be the generator in question. To figure out where we should send this generator, we first compute $\varphi_1 \circ F_* \Delta_2(F_* g)$: 
\[
\varphi_1 \circ F_* \Delta_2\left(F_* x_{0, 1} x_{1, 1} y_{0, 1}^{3} y_{1, 1}^4 \otimes x_{1, 2}^4 y_{0, 2}^3 y_{1, 2}\right) = \varphi_e\left(x_0 x_1^5 y_0^6 y_1^5\right) = x_1 y_0 y_1
\]
Now, $F_* g \not \in F_* R^{\otimes 2}$, as $x_0x_1y_0^3 y_1^4 \not \in R$, but there are certainly many $S$-multiples of $F_* g$ that land in $F_* R^{\otimes 2}$. Wherever we send $F_* g$, we need to make sure that these $S$-multiples get sent to  $R^{\otimes 2}$. 

Luckily, as described in \autoref{rem.generators}, the multiples of $F_* g$ that appear in $F_* R^{\otimes 2}$ have a very precise form. The point is that the monomial $F^e_* x_{0, 1} x_{1, 1} y_{0, 1}^{3} y_{1, 1}^4$ has a surplus of 5 more $y$'s than $x$'s. To multiply this monomial into $F^e_* R$, we must balance this out by multiplying by one more $x$ relative to the number of $y$'s (which becomes a surplus of 5 more $x$'s than $y$'s once we move them across the $F_*$). 

So for instance, $x_{0,1} \otimes 1 \cdot F_* x_{0, 1} x_{1, 1} y_{0, 1}^{3} y_{1, 1}^4 \otimes x_{1, 2}^4 y_{0, 2}^3 y_{1, 2} \in F_* R^{\otimes 2}$. This means that, if we set
\[
  \widehat{\varphi}_e\left(F_* x_{0, 1} x_{1, 1} y_{0, 1}^{3} y_{1, 1}^4 \otimes x_{1, 2}^4 y_{0, 2}^3 y_{1, 2}\right) = x_{0, 1}^{c_{0,1}} x_{1, 1}^{c_{1,1}} y_{0, 1}^{d_{0, 1}} y_{1, 1}^{d_{1,1}} \otimes x_{0, 2}^{c_{0,2}} x_{1, 2}^{c_{1,2}} y_{0, 2}^{d_{0, 2}} y_{1, 2}^{d_{1,2}}
\]
we must have 
\[
  1 + c_{0,1} + c_{1, 1} = d_{0,1} + d_{1, 1}, \textrm{ and } c_{0,2} + c_{1, 2} = d_{0,2} + d_{1, 2}.
\]
In other words, $\widehat{\varphi}_1 (F_* g)$ needs to have one more $y$ than it does $x$'s in the first tensor factor and the same number of $x$'s and $y$'s in the second tensor factor. We do this by ``taking'' one of the $y$'s from the product $x_1 y_0 y_1$ (it does not matter which) and ``giving'' it to the first tensor factor of $\widehat{\varphi}_1(F_* g)$. For instance, we can set the first tensor factor of $\widehat{\varphi}_1(F_* g)$ to be $y_0$. Then we give the rest of the product $x_1 y_0 y_1$ to the second tensor factor. At the end of the day we have
\[
  \widehat{\varphi}_1(F_* g) = y_{0,1} \otimes x_{1,2} y_{1,2}
\]
and we see that $x_{0,1}\widehat{\varphi}_1(F_* g) \in R^{\otimes 2}$ and $\Delta_2 \circ \widehat{\varphi}_1(F_* g) = x_1y_0y_1$, as desired. $\qed$

\end{example}

In what follows, we use the same technique as in the above example, but in a more general setting. We go through each tensor factor of the generator $F^e_* g$ and we ask: does it have more $y$'s than $x$'s? If so, we take the correct number of $y$'s from $\varphi_e \circ F^e_* \Delta_n(F^e_* g)$ and give them to the corresponding tensor factor of $\widehat{\varphi}_e(F^e_* g)$. Similarly, if that tensor factor of $F^e_* g$ has more $x$'s, we take the correct number of $x$'s from $\varphi_e \circ F^e_* \Delta_n(F^e_* g)$ and give them to the corresponding tensor factor of $\widehat{\varphi}_e(F^e_* g)$. The fact that $\varphi_e \circ F^e_* \Delta_n(F^e_* g)$ will always have enough $x$'s and $y$'s to do this process is expressed by \autoref{eqn.enoughVars}. The fact that, after removing these $x$'s and $y$'s, whatever is left of $\varphi_e \circ F^e_* \Delta_n(F^e_* g)$ will be an element of $R$ is expressed by \autoref{eqn.thetaisbalanced}. We can then tack on these left-overs to any tensor factor of $\widehat{\varphi}_e(F^e_* g)$ to ensure that we have $\Delta_n \circ \widehat{\varphi}_e(F^e_* g) =\varphi_e \circ F^e_* \Delta_n(F^e_* g)$. 

Recall that $\upsilon(x) = \lfloor x/q \rfloor$.

\begin{lemma} \label{claim.mainclaim}
 Let $F^e_* \prod_k \exesi{a}{k} \whysi{b}{k}$ be an $S^{\otimes n}$-module generator of $F^e_* S^{\otimes n}$ satisfying condition \autoref{eqn.Lala}, and suppose $b_k - a_k \equiv 0 \mod q$ for all $k$ with $1 \leq k \leq n$. Then
  \begin{equation}
  \label{eqn.thetaisbalanced}
   \sum_{k=1}^n b_k - a_k = q \left(\sum_{j=0}^s \upsilon\left( \sum_{k=1}^n b_{j, k} \right) - \sum_{i=0}^r \upsilon\left( \sum_{k=1}^n a_{i, k} \right) \right). 
  \end{equation}
  Moreover, setting 
  \[
    (\mu_{+,k}, \nu_{+,k}) = \begin{cases}
      \big((b_k - a_k)/q, 0 \big), & b_k - a_k \geq 0, \\
      \big(0, (a_k-b_k)/q\big), &   b_k - a_k < 0
    \end{cases}
  \]
  we have
  \begin{equation}
  \sum_{j=0}^s \upsilon\left( \sum_{k=1}^n b_{j, k} \right) \geq \sum_{k=1}^n \mu_{+,k} \eqqcolon \mu_+, \quad \sum_{i=0}^r \upsilon\left( \sum_{k=1}^n a_{i, k} \right) \geq \sum_{k=1}^n \nu_{+,k} \eqqcolon \nu_+
  \label{eqn.enoughVars}.
  \end{equation}
\end{lemma}
Assuming this lemma, we define 
\[
  \widehat{\varphi}_e \left( F^e_* \prod_k \exesi{a}{k} \whysi{b}{k}\right) = \vartheta \cdot  \prod_{k=1}^n \vartheta_k
\]
where $\vartheta_k \in \kay[\bm x_k, \bm y_k] \subset S^{\otimes n}$ is defined inductively as follows. 

For $\vartheta_1$, if $b_1 - a_1 \geq 0$ then $b_1 - a_1 = \mu_{+,1} q$. Let $f_1 = 1$ and let $g_1$ be some factor of $\whystothe{\upsilon\left(\sumofbs{\bullet} \right)}$ of degree $\mu_{+,1}$. This is possible by \autoref{eqn.enoughVars}, as $\sum_{j=0}^s \upsilon \bigl( \sum_{k=1}^n b_{j, k}\bigr) \geq \mu_{+,1}$. For all $k$, let $\varpi_k: S \to S^{\otimes n}$ be the canonical homomorphism that sends $S$ to the $k$-th factor of the tensor product. Then $\vartheta_1 = \varpi_1(g_1)$. 

Similarly, if $b_1 - a_1 < 0$, we know that $a_1 - b_1 = \nu_{+,1} q$. We let $f_1$ be some factor of 
$\exestothe{\upsilon\left(\sumofas{\bullet} \right)}$ of degree $\nu_{+,1}$ and let $g_1 = 1$. This is again possible by \autoref{eqn.enoughVars}. Then we define $\vartheta_1 = \varpi_1(f_1)$. 

Having defined $\vartheta_k, f_k$, and $g_k$ for $i = 1, \ldots, m$. We define $\vartheta_{m+1}$ as follows: if $b_{m+1} - a_{m+1} \geq 0$ then let $f_{m+1} = 1$ and let $g_{m+1}$ be some factor of 
\[
\whystothe{\upsilon\left(\sumofbs{\bullet} \right)} \Big/g_1 \cdots g_m
\] 
of degree $\mu_{+,m+1}$. We know that this is always possible by \autoref{eqn.enoughVars}. Then $\vartheta_{m+1} = \varpi_{m+1}(g_{m+1})$.  Similarly, if $b_{m+1} - a_{m+1} < 0$, we let $f_{m+1}$ be some factor of 
\[
\exestothe{\upsilon\left(\sumofas{\bullet} \right)} \Big/ f_1 \cdots f_m
\]
of degree $\nu_{+,m+1}$ and let $g_{m+1} = 1$. Then $\vartheta_{m+1} = \varpi_{m+1}(f_{m+1})$. 

Having defined $\vartheta_k$ for $k = 1, \ldots, n$, we simply let 
\[
 \vartheta = \varpi_1\left( \exestothe{\upsilon\left(\sumofas{\bullet} \right)} \whystothe{\upsilon\left(\sumofbs{\bullet} \right)} \Big/f_1 \cdots f_n g_1 \cdots g_n \right)
\]
it is clear from the definition of $\psi$ that $\widehat \varphi_e$ satisfies
\[
  \Delta_n \circ \widehat\varphi_e = \psi. 
\]

It remains to check that $\widehat \varphi_e\left( F^e_* R^{\otimes n} \right)\subset R^{\otimes n}$. It is enough to check that $\liftofphi$ sends each of the $R^{\otimes n}$-module generators from \autoref{rem.generators} to $R^{\otimes n}$. Recall any such generator can be written as

\[
 \prod_{k=1}^n \zeesi{\rho}{k} \cdot F^e_* \prod_{k=1}^n \exesi{a}{k} \whysi{b}{k}
\]
where, for each $k$, $\zeesi{\rho}{k} = \exesi{\mu}{k}$ if $b_k - a_k \geq 0$, and $\zeesi{\rho}{k} = \whysi{\nu}{k}$ if $b_k - a_k < 0$. Here, as in \autoref{rem.generators}, $\sum_{i} \mu_{i,k} = (b_k - a_k)/q$ and $\sum_{j} \nu_{j,k} = (a_k - b_k)/q$. Note that $\sum_{i} \mu_{i,k}$ and  $\sum_{j} \nu_{j,k}$ are respectively the quantities $\mu_{+. k}$ and $\nu_{+,k}$ defined in \autoref{claim.mainclaim}. Then 
\[
  \liftofphi\left( \prod_{k=1}^n \zeesi{\rho}{k} \cdot F^e_* \prod_{k=1}^n \exesi{a}{k} \whysi{b}{k} \right) =  \prod_{k=1}^n \zeesi{\rho}{k}  \cdot \liftofphi\left( F^e_* \prod_{k=1}^n \exesi{a}{k} \whysi{b}{k} \right) = 0
\]
if condition \autoref{eqn.Lala} is not satisfied by $ \prod_{k=1}^n \exesi{a}{k} \whysi{b}{k}$. Otherwise,  
\begin{align*}
   \liftofphi\left( \prod_{k=1}^n \zeesi{\rho}{k} \cdot F^e_* \prod_{k=1}^n \exesi{a}{k} \whysi{b}{k} \right) &=  \prod_{k=1}^n \zeesi{\rho}{k}  \cdot \liftofphi\left( F^e_* \prod_{k=1}^n \exesi{a}{k} \whysi{b}{k} \right)\\
   &= \prod_{k=1}^n \zeesi{\rho}{k} \cdot \vartheta \cdot \prod_{k=1}^n \vartheta_k \\
   &= \vartheta \cdot \prod_{k=1}^n \zeesi{\rho}{k} \vartheta_k.
\end{align*}
For each $k$, if $b_k - a_k \geq 0$, then $\zeesi{\rho}{k} = \exesi{\mu}{k}$ and $\vartheta_k$ is a monomial in $\{y_{0, k}, \ldots, y_{s, k} \}$ of degree $\mu_{+,k}$. Similarly, if $b_k - a_k < 0$, then by construciton $\zeesi{\rho}{k} = \whysi{\nu}{k}$ and $\vartheta_k$ is a monomial in $\{x_{0, k}, \ldots, x_{r, k} \}$ of degree $\nu_{+,k}$. In either case, we see that 
\[
\prod_{k=1}^n \zeesi{\rho}{k} \vartheta_k \in R^{\otimes n}.
\]
So it just remains to show that $\vartheta\in R^{\otimes n}$. To see this, it suffices to show that 
\[
  \exestothe{\upsilon\left(\sumofas{\bullet} \right)} \whystothe{\upsilon\left(\sumofbs{\bullet} \right)} \Big/f_1 \cdots f_n g_1 \cdots g_n  \in R.
\]
That is what \autoref{eqn.thetaisbalanced} is all about, for the degrees in terms of $y$'s and $x$'s in this monomial are, respectively,
\[
  \sum_{j=0}^s \upsilon\left( \sum_k b_{j, k} \right) - \sum_k \nu_{+,k}, \quad  \sum_{i=0}^r \upsilon\left( \sum_k a_{i, k} \right) - \sum_k \mu_{+,k}.
\]
In order to prove these two numbers are equal, it suffices to show that
\[
\sum_{j=0}^s \upsilon\left( \sum_k b_{j, k} \right) - \sum_{i=0}^r \upsilon\left( \sum_k a_{i, k} \right) = \sum_k \nu_{+,k} - \sum_k \mu_{+,k}.
\]
However, the right-hand side is nothing but $\sum_k (b_k - a_k)/q$, so this follows from \autoref{eqn.thetaisbalanced}. This shows that $\widehat \varphi_e\left( F^e_* R^{\otimes n} \right)\subset R^{\otimes n}$.

All that remains now is to prove \autoref{claim.mainclaim}.

\begin{proof}[Proof of \autoref{claim.mainclaim}]
To prove \autoref{eqn.thetaisbalanced}, we just switch the order of summation:

\begin{equation*}
  \sum_{k=1}^n b_k - a_k = \sum_{k=1}^n \left( \sum_{j=0}^s b_{j, k}  - \sum_{i=0}^r a_{i, k} \right) = \sum_{j=0}^s \sum_{k=1}^n b_{j, k} - \sum_{i=0}^r \sum_{k=1}^n a_{i,k}.
\end{equation*}
As \autoref{eqn.Lala} holds, we know that, for $i, j \geq 1$, we have $\sum_{k=1}^n b_{j, k} = q \upsilon \left( \sum_{k=1}^n b_{j, k} \right)$ and $\sum_{k=1}^n a_{i, k} = q \upsilon \left( \sum_{k=1}^n a_{i,k} \right)$. On the other hand, for $i = j = 0$, we rather have $\sum_{k=1}^n b_{j, k} = q \upsilon \left( \sum_{k=1}^n b_{j, k} \right) + 1$ and $\sum_{k=1}^n a_{i, k} = q \upsilon \left( \sum_{k=1}^n a_{i, k} \right) + 1$. In particular, for all $i$ and $j$ we have
\[
\sum_{k=1}^n b_{j, k} - \sum_{k=1}^n a_{i, k} = q \left( \upsilon \left( \sum_{k=1}^n b_{j, k} \right) - \upsilon \left( \sum_{k=1}^n a_{i, k} \right) \right)
\]
which finishes the proof of equation \autoref{eqn.thetaisbalanced}.

To prove \autoref{eqn.enoughVars}, it is enough to show
\[
   q\sum_{j=0}^s \upsilon\left( \sum_{k=1}^n b_{j, k} \right) \geq q\sum_{k=1}^n \mu_{+,k}
\]
(by symmetry, we will not have to check the other inequality). To see this, note that, by condition \autoref{eqn.Lala} we have
\[
q\sum_{j=0}^s \upsilon\left( \sum_{k=1}^n b_{j, k} \right) = \sum_{k=1}^n \sum_{j=0}^s b_{j, k} - 1.
\]
Further, we have
\begin{equation*}
q\sum_{k=1}^n \mu_{+,k} \leq \sum_{k=1}^n |b_k -a_k| \leq \sum_{k=1}^n b_k =\sum_{k=1}^n \sum_{j=0}^s b_{j, k}.
\end{equation*}
However, by condition \autoref{eqn.Lala}, we see that
\[
\sum_{k=1}^n \sum_{j=0}^s b_{j, k} \equiv 1 \mod q
\]
so we have
\[
q\sum_{k=1}^n \mu_{+,k} \neq \sum_{k=1}^n \sum_{j=0}^s b_{j, k}
\]
Thus, 
\[
q\sum_{k=1}^n \mu_{+,k} \leq \sum_{k=1}^n \sum_{j=0}^s b_{j, k} - 1 = q\sum_{j=0}^s \upsilon\left( \sum_{k=1}^n b_{j, k} \right).
\]
\end{proof}
This proves \autoref{claim.main} and therefore \autoref{thm.SegreDFR}.

\section{USTP for KLT complex singularities of diagonal $F$-regular type}
Let $R$ be a ring of equicharacteristic $0$. A \emph{descent datum} is a finitely generated $\bZ$-algebra $A\subset K$. A \emph{model} of $R$ for this descent datum is  an $A$-algebra $R_A \subset R$, such that $R_A$ is a free $A$-module and $R_A \otimes_A K = R$; see for instance \cite{HochsterHunekeTightClosureInEqualCharactersticZero} or \cite[Remark 5.3]{SmolkinSubadditivity}. Note that $A/\mu$ is a finite field, and in particular a perfect field of positive characteristic, for all maximal ideals $\mu \subset A$. We say that  $R$ is of \emph{diagonally $F$-regular type} if, for all choices of descent data $A\subset K$, the set
\[
  \{\mu \in \operatorname{MaxSpec} A \mid R_A \otimes_A A/\mu \text{ is diagonally $F$-regular}\}
\]
contains a dense open subset of  $\operatorname{MaxSpec} A$. In this case, we say that $R_A \otimes_A A/\mu$ is diagonally $F$-regular for $\mu$ ``sufficiently general.'' We notice that rings of diagonally $F$-regular type satisfy USTP via a standard reduction-mod $p$ argument. 

\begin{theorem}
\label{lemma.diagFregType}
Let $K$ be a field of characteristic 0 and let $R$ be a $K$-algebra essentially of finite type and of diagonally $F$-regular type. Let $d = \dim R$. Then we have $\mathfrak p^{(nd)} \subset \mathfrak p^n$ for all $n$ and all prime ideals $\mathfrak p \subset R$. 
\end{theorem}
\begin{proof}
  For any descent datum $A$, let $\mathfrak p_A = \mathfrak p \cap R_A$, $\mathfrak p_A^n = \mathfrak p^n \cap R_A$, and  $\mathfrak p_A^{(dn)} = \mathfrak p^{(dn)} \cap R_A$. It suffices to show that $\mathfrak p_A^{(dn)} \otimes_A A/\mu \subset \mathfrak p^n_A \otimes_A A/\mu$ for $\mu$ sufficiently general. We can choose a descent datum $A \subset K$ and a model $R_A \subset R$ such that:
  \begin{enumerate}
  \item $R_A \otimes_A A/\mu$ is diagonally $F$-regular, 
  \item $\mathfrak p_A \otimes_A A/\mu$ is a prime ideal, and
  \item $\mathfrak p^{(dn)}_A \otimes_A A/\mu = (\mathfrak p_A \otimes_A A/\mu)^{(dn)}$,
  \end{enumerate}
  for $\mu$ sufficiently general. For part (c), we use the facts that $\mathfrak p^{(dn)}$ is the $\mathfrak p$-primary component of $\mathfrak p^{dn}$, that taking powers of ideals commutes with descent, and that we can choose $A$ so that descent commutes with taking the primary decomposition of a given ideal. See \cite[\S 2.1]{HochsterHunekeTightClosureInEqualCharactersticZero} for details. It follows that
  \[
   \mathfrak p_A^{(dn)}\otimes_A A/\mu = (\mathfrak p \otimes_A A/\mu)^{(dn)} \subset (\mathfrak p \otimes_A A/\mu)^{n} = \mathfrak p^n_A \otimes_A A/\mu,
  \]
  as desired. 
\end{proof}

\begin{example}
The affine cone over $\bP^r_{\bC} \times_{\bC} \bP^{s}_{\bC}$ is a KLT singularity of diagonal $F$-regular type. In particular, USTP holds with uniform Swanson's exponent equal to $r+s=(r+s+1)-1$.
\end{example}

We conclude this paper by asking how varieties of diagonally $F$-regular type fit into the theory of singularities studied in birational geometry. 
\begin{question}
 Is there a characterization of complex varieties of diagonally $F$-regular type in terms of log-discrepancies?
\end{question}

\bibliographystyle{skalpha}
\bibliography{MainBib}

\end{document}